\documentclass{article}

\usepackage{amsmath}
\usepackage{amsthm}
\usepackage{amssymb,array}

\newtheorem{theorem}{Theorem}[section] 
\newtheorem{lemma}[theorem]{Lemma} 
\newtheorem{proposition}[theorem]{Proposition} 
\newtheorem{definition}{Definition}[section] 

\begin{document}

\title{A unified approach to $q$-special functions \\ of the Laplace type} 








\author{Yousuke Ohyama}









\date{ }

\maketitle

\begin{abstract}
We propose a unified approach to $q$-special functions, which are 
degenerations of basic hypergeometric functions ${}_2\varphi_1(a,b;c;q,x)$. 
We obtain a list of seven different class of  $q$-special functions: 
${}_2\varphi_1,  {}_1\varphi_1$, two different types of the  $q$-Bessel functions, 
the $q$-Hermite-Weber functions, two different types of the $q$-Airy functions. 
We show that there exist a relation between two types of the $q$-Airy functions. 

\end{abstract}





\section{Introduction} \label{s:intro}

In study of classical special functions,  unified theories help  us 
to understand  special functions.  We can study classical special functions
defined by differential equations easily by confluence of singular points 
or the method of separation of variables. 
Since we do not have such a unified theory 
for $q$-special functions, it  is difficult to decide whether two different 
special functions have relations or not.  
For example, we know three 
different $q$-Bessel functions and two different $q$-Airy functions and 
we know one relation between the first and the second $q$-Bessel function. 
But it is not easy to determine they are completely different or not in 
general.  

Recently many textbooks on $q$-special functions have been 
published such as Koekoek-Lesky-Swarttouw\cite{Koe}, but they have not shown unified theories of $q$-special functions
although they list up a huge list of special functions. 
In this paper, we give a unified theory for $q$-special functions, which 
come from degeneration of the basic hypergeometric functions ${}_2\varphi_1(a,b;c;q,x)$. 
In our viewpoint, we have essentially two different types of $q$-Bessel functions.
We also see that two types of $q$-Airy equations are essentially equivalent but 
they are different as  $q$-series.  A connection formula of $q$-Airy equations is 
recently found by T.~Morita \cite{Morita}. 
Our list is not enough to study whole of the Askey scheme.  It is a future problem 
to expand our unified theory  to include the Askey-Wilson polynomials.

 One of the most famous theory on a unified 
approach to classical special functions is 
by confluence of singularities \cite{Bocher, Klein}. 
We denote  irregular singular points of the Poincar\'e rank $r-1$ as $(r)$ 
in Figure \ref{classic}. 
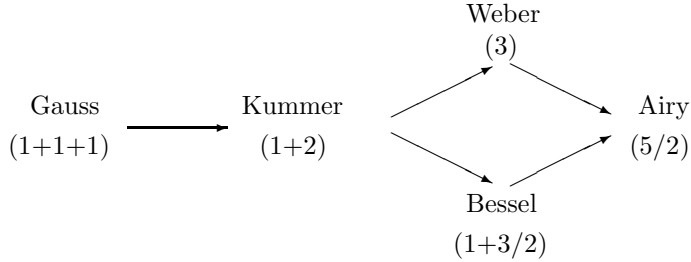
\begin{figure}[ht] \label{classic}
	\begin{picture}(300,90)(30,0)
        \put(40,52){{  Gauss} }
        \put(35,37){(1+1+1) }
        \put(120,52){{  Kummer}}
        \put(130,37){(1+2)}
        \put(205,15){{  Bessel}}
        \put(200,0){{  (1+3/2)}}
        \put(205,87){{  Weber}}
        \put(215,75){{(3)}}
        \put(270,52){{  Airy}}
        \put(270,37){{(5/2)}}
        \put(80,47){\vector(1,0){38}}
        \put(180,51){\vector(2,1){38}}
        \put(180,45){\vector(2,-1){38}}
        \put(225,25){\vector(2,1){38}}
        \put(225,71){\vector(2,-1){38}}
     \end{picture}
     \caption{The coalescent diagram of classical special functions}
\end{figure}
We 
may consider that regular singular points are singularity of the Poincar\'e rank $0$. 
We remark that we consider the Poincar\'e rank of the Bessel equation at the infinity
is 1/2, because we think ${}_0F_1(c;x)$ is a {\it true} Bessel function (see section \ref{s:two}). 
The author has learned from Professor Hideyuki  Majima that the word ``confluent'' 
appeared in the second edition of {\it Modern Analysis} \cite{Whittaker-Watson} and did not appear in the first edition  written without Watson in 1902.

Another approach is by separation of variables of the Laplacian by orthogonal 
coordinates.  This method is useful to study the Mathieu functions and the spheroidal wave 
functions \cite{Bocher, Meixner-Schafke}.  

In the study of $q$-special functions, we do not have such a unified approach 
since it is difficult to consider confluence of singular points or 
separation of variables  in the  $q$-analysis.

The third approach is classify differential equations of the Laplace type \cite{Tricomi} 
 \[
(a_0 +b_0 x) \dfrac{d^2 y}{dx^2} + (a_1 +b_1 x) \dfrac{d y}{dx} +
(a_2 +b_2 x)y=0
\] 
by means of change of variables $x\to p x+q$ and $y\to g(x) y$.  
We obtain Kummer, Weber, Bessel and Airy functions from 
equations of the Laplace type.  We review the third approach in section \ref{s:two}. 

In section \ref{s:qhyper}, we review different types of $q$-Bessel functions 
and $q$-Airy functions.  In this section we introduce two important tools to 
study $q$-difference equations. One is a shearing transformation and the second 
is  gauge transformations by $q$-products and    theta functions.

The third approach can be easily modified in $q$-difference 
equations.  We call a $q$-difference equation of the second order with 
the linear coefficients
 \[
(a_0 +b_0 x) u(xq^2)  + (a_1 +b_1 x) u(xq) +(a_2 +b_2 x)u(x)=0
\] 
{\it the hypergeometric type}.   In  section \ref{s:class}, we classify 
$q$-difference equations of the hypergeometric type and obtain 
seven types of $q$-difference equations:   

\begin{figure}[ht] \label{q-conf}
	\begin{picture}(370,70)(-10,0)
        \put(10,30){{  $ _2 \varphi_1(a,b;c;z)$ } }
        \put(110,30){{ $ _1 \varphi_1(a;c;z)$   }}
        \put(200,5){{ $ _1 \varphi_1(a;0;z)$   }}
        \put(205,55){{ $J_\nu^{(3)}$  }}
        \put(205,30){{ $J_\nu^{(1)}$ }}
        \put(270,55){{  $q$-Airy}}
        \put(270,18){{ Ramanujan}}
        \put(80,32){\vector(1,0){27}}
        \put(170,36){\vector(3,2){27}}
        \put(170,32){\vector(1,0){27}}
        \put(230,55){\vector(1,0){27}}
        \put(170,28){\vector(3,-2){27}}
        \put(255,8){\vector(2,1){14}}
        \put(230,32){\vector(3,-1){38}}
     \end{picture}
     \caption{The coalescent diagram of $q$-special functions}
\end{figure}
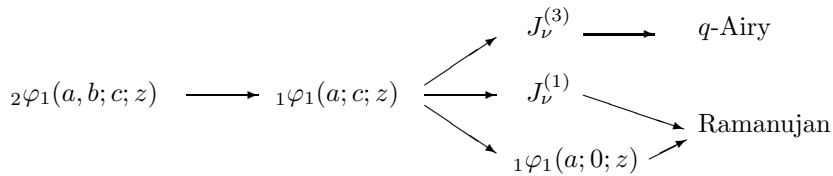

In section \ref{s:Pain}, we compare our classification of  $q$-difference equations
of the hypergeometric type to special solutions of the $q$-Painlev\'e equations.

We may consider $q$-difference equations in the matrix form 
$Y(xq)=A(x)Y(x)$.
  The matrix form has more parameters than 
a single equation of the  higher order.  These parameters 
are redundant but we can reduce the number of parameters since  the 
matrix form of $q$-difference equation have more transformations than 
the scalar form of higher order. 
For a $q$-difference equation of the hypergeometric type,  the corresponding 
matrix form is a $2\times2$ system $$Y(xq)=(A_0+A_1x)Y(x).$$
In section \ref{s:mat} we classify this type of equations.

\medskip


The author expresses his thanks to Professor Lucia Di Vizio for fruitful 
discussions when she was invited to the University of Tokyo in the winter 2011. 
The author also  thanks Professor Hidetaka Sakai,   Mr.~Yoshinori Matsumoto and  Mr.~Takeshi Morita 
for continuous discussions. 
This work is supported by the Mitsubishi foundation and the JSPS Grant-in-Aid for Scientific Research. 

\section{Differential equation of the Laplace type } \label{s:two}
In this section, we show  confluent hypergeometric series and 
classification of second order differential equations of the Laplace type. 
Although these results are already known, we review
classical results\cite{Tricomi} in order to fix our notations. 

\subsection{Confluent hypergeometric series} \label{s:two:conf} 
 We set confluent hypergeometric series found by Kummer.
\begin{eqnarray*}
{}_1F_1(a;c;z) &=&\sum^{\infty}_{n=0} 
\frac{a(a+1)\cdots(a+n-1)}{c(c+1)\cdots(c+n-1)} \frac{x^n}{n!},\\
{}_0F_1(-;c;z) &=&\sum^{\infty}_{n=0} 
\frac{1}{c(c+1)\cdots(c+n-1)} \frac{x^n}{n!}.
\end{eqnarray*}
Kummer's second confluent hypergeometric series ${}_0F_1 (-;c; z)$
satisfies the differential equation
$$x \dfrac{d^2 y}{dx^2}  + n \dfrac{d y}{dx}- y=0.$$
The infinity  is  an irregular singular point of the Poincar\'e
 rank $1/2$. 
The author does not know standard term ({\it degenerate confluent hypergeometric series} 
or {\it confluent limit hypergeometric series}) for ${}_0F_1$. 
Tricomi \cite{Tricomi} used  the notation $E_\nu(z)= 
{}_0F_1 (-;\nu+1; -z)/ \Gamma(\nu+1)$, 
Jahnke and Emde \cite{JE}  used another notation
 $\Lambda_\nu(x)= {}_0F_1 (-;\nu+1; -x^2/4)$. 

It is known that the Bessel functions are represented by 
confluent hypergeometric series in two ways. 
One is related to ${}_1F_1$:
$$J_\nu(z)=\frac{z^\nu e^{-iz}}{2^\nu \Gamma(\nu+1)} {}_1F_1
(1/2+\nu, 1+2\nu; 2iz),$$
and the second is related to ${}_0F_1$:
$$  J_\nu (2\sqrt{z}) = \frac{z^{ \nu/2}}{\Gamma(\nu+1)} {}_0F_1 (-;\nu+1; -z).
$$
 It is convenient to take ${}_0F_1 (-;c; z)$ as a
 standard form of the Bessel function at least theoretically. 
 But in many applications in mathematical physics, $J_\nu (z)$ is used since it is obtained as
 the Cylindrical functions. 
 
 \subsection{Classification of  equations of the Laplace type} \label{s:two:class-de}

 A second order differential equation of the Laplace type
\begin{equation}\label{laplace}
(a_0 +b_0 x) \dfrac{d^2 y}{dx^2} + (a_1 +b_1 x) \dfrac{d y}{dx} +
(a_2 +b_2 x)y=0
\end{equation}
is solved by special functions explained in subsection \ref{s:two:conf}.  
We remark that by the transformations ($\lambda, \mu$ and $h$ are constants)
$$x=\lambda \xi +\mu, \quad y=e^{hx}\eta,$$
$\eta(\xi)$   also satisfy a differential equations of the Laplace type \eqref{laplace}.

We set
\[
A(h)=a_0h^2+a_1 h+a_2, \quad B(h)=b_0h^2+b_1 h+b_2.
\]
Classification of second order differential equations of the Laplace type is 
as follows (see Capitolo Primo in Tricomi's book\cite{Tricomi}).
\begin{theorem}  The equation \eqref{laplace} is solved by special functions
${}_1F_1(a;c;z)$, $J_\nu(z)$, ${}_1F_1(a;1/2;z)$ and $J_{1/3}(z)$. 

(i) Assume that $b_0 \not=0$ and $\Delta:=b_1^2-4b_0b_2 \not=0$.  We set
$$h=\frac{-b_1\pm \sqrt{\Delta}}{2b_0},\  \lambda =-\frac{b_0}{B'(h)}, \ 
\mu = -\frac{a_0}{b_0},\ a= \frac{A(h)} {B'(h)},\ c=\frac{a_1b_0-a_0b_1}{b_0^2}. 
$$
Then \eqref{laplace} is solved by $\eta={}_1F_1(a;c;\xi)$, where 
$x=  \lambda\xi +\mu$ and  $y=e^{hx}\eta.$

(ii)  Assume that $b_0 \not=0$ and $\Delta = 0$.  We set
$$h=-\frac{b_1}{2b_0},\  \lambda =b_0, \ 
\mu = -\frac{a_0}{b_0},\ \alpha= \frac12-\frac{A'(h)}{2b_0},\ \beta=2\sqrt{A(h)}. 
$$
Then \eqref{laplace} is solved by $\eta=\xi^\alpha J_{2\alpha}(\beta  \sqrt{\xi})$, where 
$x=  \lambda \xi +\mu$ and  $y=e^{hx}\eta.$

(iii) Assume that $b_0 =0$ and $a_0b_1 \not=0$.  We set
$$h=-\frac{b_2}{ b_1},  \ \mu = -\frac{A'(h)}{b_1},\ 
a= \frac{A(h)} {2b_1},\ k=-\frac{ b_1}{2a_0}. $$
Then \eqref{laplace} is solved by $\eta= {}_1F_1(a;1/2; k \xi^2)$, where 
$x=    \xi +\mu$ and  $y=e^{hx}\eta.$

(iv) Assume that $b_0 =b_1=0$ and $a_0b_2 \not=0$.  We set
$$h=-\frac{a_1}{2a_0},\ \mu =  \frac{4a_0a_2-a_1^2}{4a_0b_2},
\ k= \frac{2}{3}\sqrt{\frac{b_2}{a_0}}. 
$$
Then \eqref{laplace} is solved by $\eta=\sqrt{\xi} J_{1/3}(k\xi^{3/2})$, where 
$x=    \xi +\mu$ and  $y=e^{hx}\eta.$

(v) For other cases, \eqref{laplace} is solved by elementary functions.

\end{theorem}

\section{$q$-Hypergeometric functions} \label{s:qhyper}

\subsection{Basic notations} \label{s:qhyper:basic}
We fix the notation of basic hypergeometric series.  
The $q$-Pochhammer symbol is given by 
\begin{eqnarray*}
 (a;q)_0&=& 1, \\
 (a;q)_n&=& (1-a)(1-aq )\cdots(1-aq^{n-1}), \quad (n=1,2,3,...)\\
 (a;q)_\infty&=& \prod_{k=0}^\infty (1-aq^{k-1}),  \\
 (a_1,\ldots,a_r;q)_n&=&\prod_{i=1}^n(a_i;q)_n. 
\end{eqnarray*}
The basic hypergeometric series is defined by 
\begin{equation*}
 {}_r\varphi_s\left( a_1,\ldots,a_r; b_1,\ldots,b_s;q,z\right)
=\sum_{n=0}^\infty\frac{(a_1,\ldots,a_r;q)_n}{(b_1,\ldots,b_s;q)_n(q;q)_n}
\left[(-1)^nq^{\binom{n}{2}}\right]^{1+r-s}z^n.
\end{equation*}
The theta function is defined by 
$$  
\theta_q(x):=\sum^{\infty}_{n=-\infty} q^{n(n-1)/2}x^n=(q,- x, - q/x;q)_\infty,
$$
and $\theta_q(x)$ satisfies a $q$-difference equation $x \theta_q(xq)=\theta_q(x)$.

\medskip 

The theta function $\theta_q(x)$ and the $q$-Pochhammer symbol $(a;q)_\infty$
are used to transform $q$-difference equations.  The following simple lemma plays 
an important role in this paper.

\begin{lemma}\label{guage}
We take  a $q$-difference equation
$$a(x) u(x q^2) +b(x) u(qx)+c( x) u(x)=0,$$

(1) We set $u(x)=v(x)/(sx ;q)_\infty$. Then $v(x)$ satisfies the $q$-difference equation
$$(1-sqx){a(x)} v(x q^2) +   b(x) v( x q)+ \frac{c( x)}{1-sx}  v(x)=0.$$ 

(2) We set $u(x)=\theta_q(rx) w(x)$. Then $w(x)$ satisfies the $q$-difference equation
$$\frac{a(x)}{ r q}w(x q^2) + \frac{x}{r} b(x) w(x q)+ r x^2c( x) w(x)=0.$$

\end{lemma}
We show two examples:

\noindent
 (a) For  
$$a(x) u(xq^2) +b(x) u(x q)+(c+dx) u(x)=0,$$
we set $u(x)=v(x)/(-dx/c ;q)_\infty$. Then we have
$$ \frac1c (c-dqx) a(x) v(xq^2) + b(x) v(x q)+ c  v(x)=0.$$ 

\noindent
(b) We set $u(x)=\theta_q(ax) w(x)$ for  
$$ax u(xq^2) +b(x) u(x q)+c u(x)=0.$$
Then we have
$$ w(xq^2) + b(x) w(x q)+ acx  w(x)=0.$$ 

\subsection{Shearing transformations} \label{s:qhyper:shear}

In the study of differential equations,  
shearing transformations are useful to study irregular 
singular points whose   Poincar\'e rank is non-integer. 

Shearing transformations are  also useful for $q$-differential 
equations when a slope of the Newton diagram is non-integer. 

\begin{definition}
For a $q$-difference equation 
$$a(x)u(xq^2) +b(x) u(xq)+c(x) u(x)=0,$$
 a {\it shearing transformation} is the following transformation:
$$x=t^2, \ p=\sqrt{q}, \ v(t)= u(x).$$ 
The shearing transform of the $q$-difference equation is given by
$$   a(t^2)v(tp^2) +b(t^2) v(tp)+c(t^2) v(t)=0.$$  
\end{definition}

We will show an example of a shearing transformation. 
We take a  $q$-linear equation 
 \begin{equation}\label{shear1}
(a x +b)  u(xq^2) +c u(xq)+ d u(x)=0.
\end{equation}
Here $a\not=0,b, c, d\not=0$ are constants. 
The Newton diagram of \eqref{shear1} at $x=\infty$ has a slope $1/2$. 
The shearing transform of \eqref{shear1} is
$$   (a t^2 +b)v(tp^2) +c  v(tp)+d v(t)=0.$$  
If $b\not= 0$, we set $w(t)= (p^{-1} \sqrt{-b/a}; p)_\infty v(t)$. Then $w(t)$ satisfies 
$$   a(t+ \sqrt{-b/a})w(tp^2) +c  w(tp)+ d ( t- p^{-1} \sqrt{-b/a}) w(t)=0.$$  
If $b=0$, we set $w(t)= \theta_p(t) v(t)$, as shown in subsection \ref{s:qhyper:basic}.  Then $w(t)$ satisfies 
$$   at w(tp^2) +c  w(tp)+ d t w(t)=0.$$  
In any case,  $w(t)$ has a regular singularity at  $t=\infty$ if $ad\not=0$.

\subsection{$q$-Bessel functions and $q$-Airy functions} \label{s:qhyper:bessel}
It is known that there exist three types of $q$-Bessel functions and two 
types of $q$-Airy functions.  It is not clear that relations between such functions. 
It is known that $ J^{(1)}_\nu (x;q) $ and $J^{(2)}_\nu (x;q)$ are 
essentially equivalent  and $J^{(3)}_\nu (x;q)$ is 
essentially different.   We show that two types of $q$-Airy functions 
are related by a shearing transformation and ${\rm Ai}_q(z)$ is a 
special case of $J^{(3)}_\nu (x;q)$. 

\medskip
\noindent 
1) $q$-Bessel functions:

\noindent 
It is known that there exist three types of $q$-Bessel functions 
 $J^{(1)}_\nu (x;q) $,  $J^{(2)}_\nu (x;q) $ and $J^{(3)}_\nu (x;q)$.  
 In most of literatures,   $J^{(1)}_\nu (x;q) $ and $J^{(2)}_\nu (x;q)$ are 
 called Jackson's first and second $q$-Bessel functions, and $J^{(3)}_\nu (x;q)$  is 
 called Hahn-Exton's  $q$-Bessel function.

The three $q$-Bessel functions are defined as 
\begin{eqnarray*}
 & J^{(1)}_\nu (x;q) =& \frac{(q^{\nu+1};q)_\infty}{(q ;q)_\infty} \left( 
\frac{x}{2}\right)^\nu \ {}_2\varphi_1\left( 0,0; q^{\nu+1};q, - \frac{x^2}4\right),\\
&J^{(2)}_\nu (x;q) =& \frac{(q^{\nu+1};q)_\infty}{(q ;q)_\infty} \left( 
\frac{x}{2}\right)^\nu \ {}_0\varphi_1\left( -; q^{\nu+1};q, - \frac{q^{\nu+1}x^2}4\right),\\
&J^{(3)}_\nu (x;q)  =& \frac{(q^{\nu+1};q)_\infty}{(q ;q)_\infty}\,    
x^\nu \, {}_1\varphi_1\left( 0; q^{\nu+1};q,  {q}x^2  \right). 
\end{eqnarray*}
The three $q$-Bessel functions satisfy the following 
$q$-difference equations:
\begin{eqnarray*}
& J^{(1)}_\nu: &  u(xq) -(q^{\nu/2} +q^{-\nu/2}) u(xq^{1/2}) +\left(1+\frac{x^2}4\right)u(x)=0,\\
& J^{(2)}_\nu: & \left(1+\frac{qx^2}4\right)u(xq) -(q^{\nu/2} +q^{-\nu/2}) u(xq^{1/2}) +u(x)=0,\\
& J^{(3)}_\nu: & u(xq) +[-(q^{\nu/2} +q^{-\nu/2}) +q^{\nu/2+1}x^2] u(xq^{1/2}) + u(x)=0. 
\end{eqnarray*}
It is convenient to take the inverse shearing transform of 
equations above.  We set $p=q^{1/2}$,  $t=x^2$ and $v(t)=u(x)$.  Then we have 
\begin{eqnarray*}
& J^{\prime(1)}_\nu: &  v(tp^2) -(p^{\nu} +p^{-\nu}) v(tp) +\left(1+\frac{t}4\right)v(t)=0,\\
& J^{\prime(2)}_\nu: & \left(1+\frac{p^2t}4\right)v(tp^2) -(p^{\nu} +p^{-\nu}) v(tp) +v(t)=0,\\
& J^{\prime(3)}_\nu: & v(tp^2) +[-(p^{\nu} +p^{-\nu}) +p^{\nu+2}t] v(tp) + v(t)=0. 
\end{eqnarray*}
In the following we set
\begin{eqnarray*}
& E^{(1)}_\nu(x;p)  & = t^\nu {} {}_2\varphi_1\left( 0,0; p^{2\nu+1};p, -  {x}/4\right),\\
& E^{(2)}_\nu(x;p)  & = t^\nu {} {}_0\varphi_1\left( -  ; p^{2\nu+1};p, -  {x} p^{2\nu+2}/4\right),\\
& E^{(3)}_\nu(x;p)  & = t^\nu {} {}_1\varphi_1\left( 0 ; p^{2\nu+1};p,   {x} p^{2}/4\right). 
\end{eqnarray*}
$E^{(1)}_\nu(x;p)$ and $E^{(1)}_{-\nu}(x;p) $  are solutions of the equation $J^{\prime(1)}_\nu$. 
$E^{(2)}_\nu(x;p)$ and $E^{(2)}_{-\nu}(x;p) $  are solutions of the equation $J^{\prime(2)}_\nu$. 
$E^{(3)}_\nu(x;p)$ and $E^{(3)}_{-\nu}(xp^{-2\nu};p) $  are solutions of the equation $J^{\prime(3)}_\nu$. 
We use the notation $E^{(j)}_\nu(x;p)$ since Tricomi used $E_\nu(x)$ in his study in the 
Bessel functions  (see subsection \ref{s:two:conf}). 

By Hahn's   formula
\begin{equation}\label{hahn}
J_\nu^{(2)}(x;q)= (-\frac{x^2}4;q)_\infty \cdot J_\nu^{(1)}(x;q), \quad  
E^{(2)}_\nu(x;p)= (-\frac{t}4;p^2)_\infty E^{(1)}_\nu(x;p)
\end{equation}
In this sense $J_\nu^{(1)}$ and $J_\nu^{(2)}$ are equivalent. 
Fitouhi and Hamza \cite{Fitouhi-Hamza} have defined another $q$-Bessel function 
$j_\alpha(x, q^2)$, which is essentially 
equivalent to  $J_\nu^{(3)}(x;q)$.

\bigskip 
\noindent
2) $q$-Airy functions

\noindent
It is known that there exist 
two different types of the $q$-Airy functions.  The $q$-Airy function 
\textrm{Ai}$_q(x)$ is found in 
the study of special solutions of the second $q$-Painlev\'e equation
\cite{KMNOY}.   The limit $q\to 1$ tends to the Airy function at least 
around the infinity \cite{HKW}.  The second function   $\textrm{A}_q(x)$, is called the {\it Ramanujan} function,
which  is found by Ismail in the study of 
asymptotic behavior of the $q$-Hermite polynomial \cite{Ismail}. 
\begin{eqnarray*}
\textrm{Ai}_q(x)&=& {}_1\varphi_1(0;-q;q,-x), \\
\textrm{A}_q(x)&=& {}_0\varphi_1(-;0;q,-qx).
\end{eqnarray*}
It was not known any  relation between the $q$-Airy function 
 \textrm{Ai}$_q(x)$ and  the {  Ramanujan} function $\textrm{A}_q(x)$, but recently Morita has found 
a connection formula between \textrm{Ai}$_q(x)$ and $\textrm{A}_q(x)$ 
(see \eqref{morita} in subsection \ref{s:class:c}).

The Airy function is a special case of the modified Bessel function:
$$\textrm{Ai}(x)=\frac1\pi \sqrt{\frac{x}3}\,  K_{1/3}\left(\frac{2}3x^{3/2}\right).$$ 
As the same as the differential case, 
the $q$-Airy function {Ai}$_q(x)$ is related to $J^{(3)}_\nu (x;q)$.

\begin{lemma}  
If $q^\nu=-1$, 
we have 
$$J^{(3)}_\nu (x;q)  
= \frac{(-q ;q)_\infty}{(q ;q)_\infty}   
x^\nu \, {} \mathrm {Ai}_q(-q x^2).$$
\end{lemma}

The $q$-Airy function $\textrm{Ai}_q(x)$ satisfies a $q$-difference equation 
\begin{equation}\label{q-a1}
u(xq^2)+x u(xq)- u(x)=0,
\end{equation}
and $ \textrm{A}_q(x)$  satisfies a $q$-difference equation 
\begin{equation}\label{q-a2}
qx u(xq^2)-  u(xq)+  u(x)=0.
\end{equation}
These two equation connected by shearing transformation as we has seen in subsection \ref{s:qhyper:shear}. 

\begin{lemma}  \label{airy-ram}
If $u(x)$ satisfies \eqref{q-a1}, 
$$v(x)=\theta_q(-q^2 x) \mathrm {Ai}_q(1/x)$$
satisfies an inverse shearing transform of a modified equation of \eqref{q-a1}:
$$-q^5  x^2 v(xq^2)-  v(xq)+  v(x)=0,$$
which is  is solved by $ \mathrm{A}_{q^2}(x^2q^3)$. 
\end{lemma}

We can check out the lemma above directly.

\section{Classification} \label{s:class}

We classify a $q$-difference equation of the hypergeometric type:
\begin{equation} \label{qhyper}
(a_0 +b_0 x) u(xq^2)  + (a_1 +b_1 x) u(xq) +(a_2 +b_2 x)u(x)=0.
\end{equation}
The equation above has transformations which keep the hypergeometric type. 
Such transformations ware known by Hahn \cite{Hahn}.

\subsection{A $q$-analogue of the Riemann scheme}  \label{s:class:R}
Before we list up transformations which keep the hypergeometric type, 
we introduce Matsumoto's $q$-analogue of the Riemann scheme  \cite{Mat}.

\begin{definition}
For \eqref{qhyper}, we set two characteristic polynomials 
$a_0\mu^{2}+a_1\mu+a_2=0$ and $b_0+b_1\lambda+b_2\lambda^{2}=0$
The first one is a characteristic polynomial around $x=0$ and 
the second one is a characteristic polynomial around
 $x=\infty$.  The roots of both polynomials are called 
 characteristic exponents $\mu_1, \mu_2, \lambda_1, \lambda_2$. 
The  characteristic exponents is considered as $\infty$ when 
$a_0=0$ or $b_2=0$.  We set $\rho_{1}=- {a_0}/{b_0},\quad  \rho_{2}=- {b_2}/{a_2}$, 
which are {\it virtual}\  exponents. 
\end{definition}

It is easily checked out that $\rho_{1}\rho_{2}\lambda_{1}\lambda_{2}\mu_{1}\mu_{2}=1$, 
which is a $q$-analogue of Fuchs' relation for Fuchsian differential equations. 
If all of exponents are not zero, \eqref{qhyper} is written only by exponents:
\begin{eqnarray*}
\lambda_{1}\lambda_{2}(x-\rho_{1})u(x q^{2})-\{(\lambda_{1}+\lambda_{2})x-\lambda_{1}\lambda_{2}\rho_{1}(\mu_{1}+\mu_{2})\}u(x q) +(x-\lambda_{1}\lambda_{2}\mu_{1}\mu_{2}\rho_{1})u(x)=0, 
\end{eqnarray*}
which is a $q$-analogue of Papperitz's differential equation \cite{Papperitz}. 

\begin{definition}
For \eqref{qhyper}, we set a $q$-analogue of the Riemann scheme: 
\begin{equation*}
\Phi\left\{
\begin{matrix}
0 & \infty & * \\
\mu_1 & \lambda_1 & \rho_1 &; \quad  x\\
\mu_2 & \lambda_2 & \rho_2
\end{matrix}
\right\}.
\end{equation*}
\end{definition}

The Riemann scheme represents a space of solutions.  
A $q$-analogue of the Riemann scheme has already shown by Hahn \cite{Hahn}, but
Hahn's Riemann scheme is just a table setting out coefficients $a_j, b_k$.  
 Since exponents is essential in differential or difference equations, 
 we array exponents in the scheme as the same as the original Riemann scheme.

\par \bigskip 

We study transformations which keep   hypergeometric type.  We consider two difference equations which are transformed by such transformations are equivalent.

\begin{theorem} \label{trans}   The following four transformations on $q$-difference equations
 keep the hypergeometric type. 
(1) $x \to c x$, \ (2)  ${u \to x^\gamma u}$, \ 
(3) $x \to 1/x$, \ (4) $u \to  (\rho_2x; q)_\infty/(x/\rho_1 q; q)_\infty u(x)$.
By these transformations, the Riemann scheme is transformed as follows:

(1) $x \to c x$: 
\begin{equation*}
\Phi\left\{
\begin{matrix}
0 & \infty & * \\
\mu_1 & \lambda_1 & \rho_1 &; \quad  c x\\
\mu_2 & \lambda_2 & \rho_2
\end{matrix}\right\}=
\Phi\left\{
\begin{matrix}
0 & \infty & * \\
\mu_1 & \lambda_1 & {\rho_1}/c &; \quad  x\\
\mu_2 & \lambda_2 & c\rho_2 
\end{matrix}
\right\}
\end{equation*}

(2) ${u \to x^\gamma u}$ for  $ c=  q^\gamma$: 
\begin{equation*}
\Phi\left\{
\begin{matrix}
0 & \infty & * \\
\mu_1 & \lambda_1 & \rho_1 &; \quad  x\\
\mu_2 & \lambda_2 & \rho_2
\end{matrix}\right\}
=x^\gamma
\Phi\left\{
\begin{matrix}
0 & \infty & * \\
{\mu_1}/c & c \lambda_1 & \rho_1 &; \quad  x\\
{\mu_1}/c & c \lambda_2 & \rho_2
\end{matrix}\right\}
\end{equation*}

(3) $x \to 1/x$: 
\begin{eqnarray*}
\Phi\left\{
\begin{matrix}
0 & \infty & * \\
\mu_1 & \lambda_1 & \rho_1 &; \quad  \dfrac{1}{x}\\
\mu_2 & \lambda_2 & \rho_2
\end{matrix}\right\}
&=& \Phi\left\{
\begin{matrix}
0 & \infty & * \\
\lambda_1 & \mu_1 & \rho_2 &; \quad  q^2x\\
\lambda_2 & \mu_2 & \rho_1
\end{matrix}\right\}  \\
&=& \Phi\left\{
\begin{matrix}
0 & \infty & * \\
\lambda_1 & \mu_1 & {\rho_2}q^{-2} &; \quad  x\\
\lambda_2 & \mu_2 & \rho_1q^2
\end{matrix}\right\}
\end{eqnarray*}

(4) $u \to  (\rho_2x; q)_\infty/(x/\rho_1 q; q)_\infty u$:
\begin{equation*}
\Phi\left\{
\begin{matrix}
0 & \infty & * \\
\mu_1 & \lambda_1 & \rho_1 &; \quad  x\\
\mu_2 & \lambda_2 & \rho_2
\end{matrix}\right\}
=\frac{(x/\rho_1 q; q)_\infty}{(\rho_2x; q)_\infty}
\Phi\left\{
\begin{matrix}
0 & \infty & * \\
\mu_1 & {q}/{\mu_1\mu_2\lambda_1} & {1}/{\rho_2 q} &; \quad  x\\
\mu_2 & {q}/{\mu_1\mu_2\lambda_2} & {1}/{\rho_1 q}
\end{matrix}\right\}
\end{equation*}
\end{theorem}

\bigskip 
\noindent
{\it Remark.} 
If we apply the transformation  (4) to the basic hypergeometric series 
$_2\varphi_1(a,b;c;q,x)$, we obtain Heine's transformation
\begin{equation*}\label{heine2}
_2\varphi_1(a,b;c;q,x)=\frac{(abx/c; q)_\infty}{(x; q)_\infty}{}_2\varphi_1\left(\frac{c}{a},\frac{c}{b};c;q,\frac{ab}{c}x\right).
\end{equation*}
We consider the transformations not only for individual solutions but also 
for  $q$-difference equations. 
Hahn's   formula \eqref{hahn} for the $q$-Bessel function is 
 essentially  the same as the transformation (4).  When some of the exponents 
 are zero or the infinity, we take $\theta_q(sx)$ instead of $(sx; q)_\infty$ 
 in (4) as we have seen on Lemma \ref{guage}. 

By means of transformations in Theorem \ref{trans}, we easily obtain a part of 
$q$-analogue of Kummer's twenty-four solutions of the hypergeometric equations. 
\begin{proposition}  The $q$-difference equation
 \begin{equation*}\label{qhg}
(c-abq x) u(xq^2)-[c+q-(a+b)qx]u(qx)+q(1-x) u(x)=0. 
\end{equation*} 
for the basic hypergeometric series   $_2\varphi_1(a,b;c;q,x)$ have the 
following eight solutions:
 \begin{eqnarray*}
& 
_2\varphi_1(a,b;c;q,x), 
& \dfrac{(abx/c;q)_\infty}{(x;q)_\infty}{_2\varphi_1}\left(\dfrac{c}{a},\dfrac{c}{b};c;q,\frac{ab}{c}x\right), \\
& x^{1-\gamma}{_2\varphi_1}\left(\dfrac{aq}{c},\dfrac{bq}{c};\dfrac{q^2}{c};q,x\right),&
x^{1-\gamma}\dfrac{(abx/c;q)_\infty}{(x;q)_\infty}{_2\varphi_1}\left(\frac{q}{a},\frac{q}{b};\frac{q^2}{c};q,\dfrac{ab}{c}x\right), \\ 
& x^{-\alpha}{_2\varphi_1}\left(a,\dfrac{aq}{c};\dfrac{aq}{b};q,\dfrac{cq}{abx}\right),& 
x^{-\alpha}\dfrac{(q/x;q)_\infty}{(cq/abx)_\infty}{_2\varphi_1}\left(\frac{q}{b},\frac{c}{b};\dfrac{aq}{b};q,\dfrac{q}{x}\right), \\
& x^{-\beta}{_2\varphi_1}\left(b,\dfrac{bq}{c};\dfrac{bq}{a};q,\dfrac{cq}{abx}\right),&  x^{-\beta}\dfrac{(q/x;q)_\infty}{(cq/abx;q)_\infty}{_2\varphi_1}\left(\frac{q}{a},\frac{c}{a};\frac{bq}{a};q,\frac{q}{x}\right).
\end{eqnarray*}
Here $a=q^\alpha, b=q^\beta$ and $c=q^\gamma.$
These solutions are $q$-analogue of a part of Kummer's twenty-four solutions \cite{Hahn, Mat} and tends to
  Kummer's original twenty-four solutions when $q\to 1$.
\end{proposition}

\begin{proof} We can easily obtained all of eight solutions applying 
the transformations in Theorem \ref{trans} to $_2\varphi_1(a,b;c;q,x)$. 
\end{proof}

\noindent 
{\it Remark.}  A $q$-analogue of Kummer's twenty-four solutions are 
considered by Hahn \cite{Hahn}.  
Other sixteen solutions are not represented by $_2\varphi_1$.  We need
$_2\varphi_2$ or $_3\varphi_2$ to represent a $q$-analogue of Kummer's twenty-four solutions.

\subsection{Classification of $q$-difference equations} \label{s:class:c}
By the transformations in Theorem \ref{trans}, we can classify all of 
$q$-difference equations \eqref{qhyper}  of the hypergeometric type. 
This classification gives a coalescent diagram of $q$-special functions.

\begin{theorem}\label{class}  A $q$-difference equation \eqref{qhyper}  of the hypergeometric type 
reduces to one of the following equation by transforms in theorem \ref{trans}. 

{1)} When $a_1a_3  b_1 b_3 \not=0$, { Heine's hypergeometric ${}_2\varphi_1(a,b ; c; q,x)$}:
 \begin{equation*} 
(c-abq x) u(xq^2)-[c+q-(a+b)qx]u(qx)+q(1-x) u(x)=0. 
\end{equation*}

{2)} When $b_3=0$, $a_1a_3  b_1b_2  \not=0$,  {${}_1\varphi_1(a ; c; q,x)$}:
$$(c-a q x) u(xq^2)-(c+q-qx)u(qx)+q u(x)=0.$$

3-1) When $b_1=b_2=0$, $a_3 \cdot a_2a_1  b_3 \not=0$, 
{Jackson's Bessel  functions $E_\nu^{(1)}(x;q)$}F
$$ u(xq^2)-(q^\nu+ q^{-\nu})u(xq)+(1+x/4) u(x)=0.$$

{3-2)} When $b_1=b_3=0$, $a_2 \cdot a_3 a_1b_2\not=0$, {
Hahn-Exton's Bessel  functions $E_\nu^{(3)}(x;q)$}: 
$$ u(xq^2)+ [-(q^\nu+ q^{-\nu}) +q^{2-\nu}x] u(xq)+   u(x)=0.$$

{3-3)} When $b_3=a_1=0$, $a_2b_2 \cdot a_3 b_1\not=0$,  {
$q$-Hermite-Weber ${}_1\varphi_1(a ;0; q,x)$}
$$ a  x u(xq^2) + (1- x  ) u(xq)-  u(x)=0.$$

{4-1)} When $b_1=a_2=b_3=0$, {$q$-Airy $_1\varphi_1\left( 0 ;-q  ;q,-x\right)$} :
$$ u(xq^2)+x u(xq)- u(x)=0.$$

4-2) When $a_1=b_2=b_3=0$, {the Ramanujan function $_0\varphi_1\left( - ;0  ;q,-tq\right)$}:
$$ qx u(xq^2)-  u(xq)+  u(x)=0.$$
\end{theorem}
 
We list up the Riemann scheme corresponding to the classification in the Theorem \ref{class}. 
In the following, $A$ and $B$ are quasi-constants, i.e. 
$A(qx)=A(x), B(qx)=B(x).$

\par \medskip 

(1) For $ _2\varphi_1 (a,b,;c;q,x)$, the Riemann scheme is given by
\begin{equation*}
\Phi\left\{
\begin{matrix}
0 & \infty & * \\
1 & a & {c}/{abq} &; \quad  x\\
{q}/{c} & b & 1
\end{matrix}
\right\}.
\end{equation*}
The solution space $\Phi$  is
$$ 
\Phi= A\, {}_2\varphi_1(a,b;c;q,x) + B \, 
x^{1-\gamma}{_2\varphi_1}\left(\frac{aq}{c},\frac{bq}{c};\frac{q^2}{c};q,x\right)
$$
around $x=0$. Here $q^\gamma = c$. 

\par \medskip 

(2) For $ _1\varphi_1 (a ;c;q;x)$, the Riemann scheme is given by
\begin{equation*}
\Phi\left\{
\begin{matrix}
0 & \infty & * \\
1 & a & {c}/{aq} &; \quad  x\\
{q}/{c} & \infty & 0
\end{matrix}
\right\}.
\end{equation*}
The solution space $\Phi$  is
$$ 
\Phi= A\, {}_1\varphi_1(a  ;c;q,x) + B \, 
x^{1-\gamma}{_1\varphi_1}\left(\frac{aq}{c};\frac{q^2}{c};q,x\right)
$$
around $x=0$. Here $q^\gamma = c$.

\par \medskip 

(3-1)  The Riemann scheme is given by
\begin{equation*}
\Phi\left\{
\begin{matrix}
0 & \infty & * \\
q^\nu & 0 & \infty &; \quad  x\\
q^{-\nu} & 0  & -1/4
\end{matrix}
\right\}.
\end{equation*}
The solution space $\Phi$  is
$$ 
\Phi= A\, x^\nu{}_2\varphi_1(0,0;q^{1+\nu};q,-x/4) + B \, 
 x^{-\nu}{}_2\varphi_1(0,0;q^{1-\nu};p,-x/4)
$$
around $x=0$. 

The standard form of Jackson's first $q$-Bessel function is 
\[
J^{(1)}_\nu (x;q) =  \frac{(q^{\nu+1};q)_\infty}{(q ;q)_\infty} \left( 
\frac{x}{2}\right)^\nu \ {}_2\varphi_1\left( 0,0; q^{\nu+1};q,- \frac{x^2}4\right). 
\]
Solutions of the $q$-linear equation
\begin{equation}\label{qb1}
u( x q^2) - (q^{\nu/2} + q^{-\nu/2} ) u(x q) + \left( 1 + \frac{x^2}{4} \right)  u(x)=0
\end{equation}
is $u= A\, J^{(1)}_\nu (x;q) +B \, J^{(1)}_{-\nu} (x;q). $  The equation \eqref{qb1} is obtained 
by a shearing transformation of (3-1). 
 
\par \medskip 

(3-2) The Riemann scheme is given by
\begin{equation*}
\Phi\left\{
\begin{matrix}
0 & \infty & * \\
p^{\nu} & 0 & \infty &; \quad  x\\
p^{-\nu} & \infty & 0
\end{matrix}
\right\}.
\end{equation*}
The solution space $\Phi$  is
\[
\Phi=A\, x^\nu \, {}_1\varphi_1\left( 0; q^{1+2\nu };q,  {q^2}x   \right) +
B \,  x^{-\nu} \, {}_1\varphi_1\left( 0; q^{1-2\nu};q,  {q^{2-2\nu}}x   \right).
\]

\par \medskip 

(3-3) The Riemann scheme is given by
\begin{equation*}
\Phi\left\{
\begin{matrix}
0 & \infty & * \\
1 & a & 0 &; \quad  x\\
\infty & \infty & 0
\end{matrix}
\right\}.
\end{equation*}
The solution space $\Phi$  is
\[
\Phi=A\, {}_1\varphi_0(a;-;q, x) +B \, \theta(-ax/q) {}_2\varphi_0(q/a,0;-;q, ax/q^2).
\]

\par \medskip

(4-1) The Riemann scheme is given by
\begin{equation*}
\Phi\left\{
\begin{matrix}
0 & \infty & * \\
1 & 0 & \infty &; \quad  x\\
-1 & \infty & 0
\end{matrix}
\right\}.
\end{equation*}
The solution space $\Phi$  is
\[
\Phi=A\, \textrm{Ai}_q(x) +B \,  e^{\pi i \textrm{lq}\,x} {}\textrm{Ai}_q(-x),
\]
where $\textrm{lq}\,x =\log_e x/ \log_e q$.

(4-2) The Riemann scheme is given by
\begin{equation*}
\Phi\left\{
\begin{matrix}
0 & \infty & * \\
1 & \infty & 0 &; \quad  x\\
\infty & \infty & 0
\end{matrix}
\right\}.
\end{equation*}
The solution space $\Phi$  is
\[
\Phi=A\, {}_0\varphi_1(-;0;q,-qx) +B \,  {\theta_q(x)} \, {}_2\varphi_0(0,0;-;q,-x/q). 
\]

The Newton diagram of equations in Theorem \ref{class} is as follows. The black circle means 
a coefficient which is not zero. The Newton diagram of (3-2) and (4-1) are the same, because the $q$-Airy function is 
a special case of $E^{(1)}_\nu (x;q)$ in case $q^\nu=\pm i$.  The Newton diagram 
explains that the coalescent diagram in Figure \ref{q-conf}.  If   the Newton diagram of a
 $q$-difference equation is a subset of the the Newton diagram of another 
 $q$-difference equation, we may take a suitable limit. 

\begin{center}
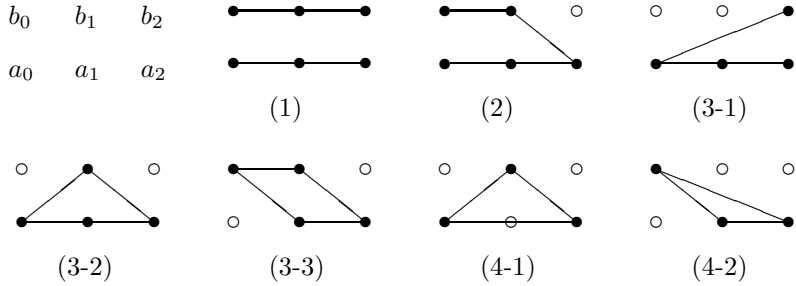
\begin{figure}[ht] 
	\begin{picture} (300, 120) (-20,0)
	
	\put( -5,75){$a_0$} \put(20,75){$a_1$}  \put(45, 75){$a_2$} 
	\put( -5,95){$b_0$} \put(20,95){$b_1$}  \put(45, 95){$b_2$} 
    \put(10,60){{  }}
    
    \put(80,80){\circle*{4}} \put(105,80){\circle*{4}}  \put(130, 80){\circle*{4}} 
	\put(80,100){\circle*{4}} \put(105,100){\circle*{4}}  \put(130, 100){\circle*{4}} 
	\put(82,80){\line(1,0){21}}\put(107,80){\line(1,0){21}} 
	\put(82,100){\line(1,0){21}}\put(107,100){\line(1,0){21}}   
    \put(90,60){{ (1)  }}

    \put(160,80){\circle*{4}} \put(185,80){\circle*{4}}  \put(210, 80){\circle*{4}} 
	\put(160,100){\circle*{4}} \put(185,100){\circle*{4}}  \put(210, 100){\circle{4}} 
	\put(162,80){\line(1,0){21}} \put(187,80){\line(1,0){21}} 
	\put(162,100){\line(1,0){21}} \put(185,100){\line(5,-4){25}} 
    \put(170,60){{ (2)  }}
    
    \put(240,80){\circle*{4}} \put(265,80){\circle*{4}}  \put(290, 80){\circle*{4}} 
	\put(240,100){\circle{4}} \put(265,100){\circle{4}}  \put(290, 100){\circle*{4}} 
	\put(242,80){\line(1,0){21}} \put(267,80){\line(1,0){21}} 
	\put(240,80){\line(5,2){50}} 
    \put(250,60){{ (3-1)  }}

	\put( 0,20){\circle*{4}} \put(25,20){\circle*{4}}  \put(50, 20){\circle*{4}} 
	\put( 0,40){\circle{4}} \put(25,40){\circle*{4}}  \put(50, 40){\circle{4}} 
	\put( 0,20){\line(1,0){25}} \put(25,20){\line(1,0){25}} 
	\put( 0,20){\line(5,4){25}} \put(25,40){\line(5,-4){25}}
    \put(10,0){{ (3-2)  }}
    
    \put(80,20){\circle{4}} \put(105,20){\circle*{4}}  \put(130, 20){\circle*{4}} 
	\put(80,40){\circle*{4}} \put(105,40){\circle*{4}}  \put(130, 40){\circle{4}} 
	\put(80,40){\line(1,0){25}} \put(105,20){\line(1,0){25}} 
	\put(80,40){\line(5,-4){25}} \put(105,40){\line(5,-4){25}}
    \put(90,0){{ (3-3)  }}

    \put(160,20){\circle*{4}} \put(185,20){\circle{4}}  \put(210, 20){\circle*{4}} 
	\put(160,40){\circle{4}} \put(185,40){\circle*{4}}  \put(210, 40){\circle{4}} 
	\put(160,20){\line(1,0){50}}  
	\put(160,20){\line(5,4){25}} \put(185,40){\line(5,-4){25}}
    \put(170,0){{ (4-1)  }}
    
    \put(240,20){\circle{4}} \put(265,20){\circle*{4}}  \put(290, 20){\circle*{4}} 
	\put(240,40){\circle*{4}} \put(265,40){\circle{4}}  \put(290, 40){\circle{4}}
	\put(240,40){\line(5,-4){25}}  
	\put(240,40){\line(5,-2){50}} \put(265,20){\line(1,0){25}} 
    \put(250,0){{ (4-2)  }}
	
     \end{picture}
     \caption{The Newton diagram}
\end{figure}
\end{center}

\begin{lemma}
In Theorem \ref{class}, 
  {(4-2)} is equivalent to (4.1) by  {shearing transformation}.
\end{lemma}
This lemma shows that there exists a relation between the $q$-Airy function and the Ramanujan function. 
Recently, Morita  \cite{Morita} has shown a connection formula between $\mathrm {Ai}_q (x)$ and $\mathrm {A }_{q}(x)$:
\begin{equation}\label{morita}
\mathrm {A }_{q^2}(-q^3/x^2) =  \frac{q^2}{(q,-1;q)_\infty} 
\left( -\theta(-x/q) \mathrm {Ai}_q (xq^2) + {\theta(x/q) } \mathrm {Ai}_q (-xq^2)  \right).
\end{equation}

\section{Hypergeometric solutions of the $q$-Painlev\'e equations}\label{s:Pain}

As the same as the Painlev\'e differential equations have particular solutions 
represented by (confluent) hypergeometric functions, the $q$-Painlev\'e equations
also have special solutions written by $q$-hypergeometric functions. 

Kajiwara,   Masuda,   Noumi,  Ohta, and  Yamada has studied $q$-hypergeometric solutions of the $q$-Painlev\'e equations\cite{KMNOY}.
The degeneration diagram of $q$-hypergeometric solutions of 
the $q$-Painlev\'e equations is as follows:
\[
\begin{array}{cccccccccc}
q$-${\bf P}&q\textrm{-}P_\mathrm{VI}  & \rightarrow & q\textrm{-}P_\mathrm{V} &\rightarrow & 
\begin{array}{c}q\textrm{-}P_\mathrm{IV}\\
q\textrm{-}P_\mathrm{III}\\
\end{array}
& \rightarrow & 
q\textrm{-}P_\mathrm{II}& \rightarrow &  q\textrm{-}P_\mathrm{I}\\
\\
\mathrm{\bf HG} & {\displaystyle {}_2\varphi_1}& \rightarrow
& {\displaystyle {}_1\varphi_1} & \rightarrow
& {\displaystyle {{}_1\varphi_1\left( a;0  ;q, z\right)}}\atop
{\displaystyle {{}_1\varphi_1\left(  0; b  ;q, z\right)}} & \rightarrow
& {\displaystyle {}_1\varphi_1\left( 0; -q  ;q, z\right)}& 
\rightarrow & {\mathrm {none}} \\  \\ 
&  {(1)} & \rightarrow
&  {(2)}  & \rightarrow
&  \displaystyle \textrm{(3-3)}  \atop
 \displaystyle \textrm{(3-2)} & \rightarrow
&  \textrm{(4-1)} & 
\rightarrow & \mathrm{none}
\end{array}
\]
Comparing our list in Theorem \ref{class}, we do not have (3-1) and (4-2). 
The equation (4-2) is related to (4-1)  by a shearing transformation. 
The equation (3-1) appears in  another form of $q$-$P_{\rm III}$.  

It is known that there are several types of the $q$-Painlev\'e equations, whose 
limit $q\to 1$ tends to the same Painlev\'e differential equation. 
It is widely known that there exist two different types of $q$-$P_{\rm III}$.  One is  
 called $q$-$P_{\textrm{III}}(A_5^{(1)})$
by Sakai \cite{S}:
\begin{equation*}
\frac{y\bar{y}}{a_3a_4}
 =-\frac{\bar{z}(\bar{z}-b_2t)}
{\bar{z}-b_3}, \quad
\frac{z\bar{z}}{b_3}
=-\frac{y(y-a_1t)}{a_4(y-a_3)}.
\end{equation*}
Here the affine root system $(A_5^{(1)})$ means the type of the initial value space 
of the $q$-Painlev\'e equation, which is completely classified by Sakai \cite{S}. 

Another equation  is shown by Ramani, Grammaticos  and Hietarinta \cite{RGH91}: 
\begin{equation}\label{RGH}\frac{\overline{w} \underline{w}}{a_3a_4 }
= \frac{(w-a_1 s)(w- s a_2)}{(w- a_1)(w-a_4)},\end{equation}
which is a {symmetric specialization}  of $q$-$P_{\rm VI}$ found by Jimbo and Sakai \cite{JS96}.

Jimbo and Sakai found $q$-$P_{\rm VI}$ as a connection preserving deformation of a linear
$q$-difference equation.  Their $q$-$P_{\rm VI}(A_3^{(1)})$ is given by
\begin{align*}
\frac{y\bar{y}}{a_3a_4}
&=\frac{(\bar{z}-b_1t)(\bar{z}-b_2t)}{(\bar{z}-b_3)(\bar{z}-b_4)},\quad
\frac{z\bar{z}}{b_3b_4}
 =\frac{(y-a_1t)(y-a_2t)}{(y-a_3)(y-a_4)}, \quad 
\left(\frac{b_1b_2}{b_3b_4}=q\frac{a_1a_2}{a_3a_4}\right).
\end{align*}
 For $q$-$P_{\rm VI}(A_3^{(1)})$, we set 
\begin{align*} t=s^2, \   q=p^2, \  
       & b_i= p a_i \  (i=1,2), \     b_i=  a_i \  (i=3,4), \\ 
       & y(t)=\overline{w}(s), \  z(t)=w(s). 
\end{align*}
Then we obtain the equation \eqref{RGH}.  In this sense, 
\eqref{RGH}  as can be considered as a {symmetric specialization}  of $q$-$P_{\rm VI}$. 
Since the  type of the initial value space of \eqref{RGH} is the same as 
the  type of $q$-$P_{\rm VI}(A_3^{(1)})$, we may denote \eqref{RGH} as $q$-$P_{\textrm{III}}(A_3^{(1)})$. 
Kajiwara,   Ohta and  Satsuma\cite{KOS} has shown that 
$J_\nu^{(1)}(x;q)$ is a special solution of $q$-$P_{\textrm{III}}(A_3^{(1)})$.  
This fact  is quite 
natural since $J_\nu^{(1)}(x;q)$ can be represented by a basic hypergeometric series ${}_2\varphi_1(p^{\nu+1/2},- p^{\nu+1/2}; p^{2\nu+1 };p,x) $, where $ p=\sqrt{q} $. 

\begin{theorem}
In the degeneration scheme of $q$-special functions in Theorem \ref{class}, six types 
of $q$-special functions  appear as special solutions of the $q$-Painlev\'e equations. 
Type (3-1) is appeared as special solutions of a symmetric specialization  of $q$-$P_{\rm VI}$. 
Type (4-2) does not appear as special solutions of the $q$-Painlev\'e equations, but 
a shearing transform of (4-2) is equivalent to the type (4-1), which is 
appeared as special solitons of $q$-$P_{\rm II}$.
\end{theorem}

\section{ Matrix Form} \label{s:mat}
 In $2\times 2$ matrix form, a $q$-difference  equations of the hypergeometric type is 
 given by
 \begin{equation}\label{matr}
Y(xq)=A(x)Y(x)=\left(A_0+A_1x\right)Y(x)
\end{equation}  
Here 
 $$Y(x)=\begin{pmatrix}
y_1(x ) \\ y_2(x )
\end{pmatrix}, \quad   A_0 = 
 \begin{pmatrix}
a_{11} &  a_{12} \\
a_{21} &  a_{22}
\end{pmatrix}, \quad 
A_1=\begin{pmatrix}
b_{11} &  b_{12} \\
b_{21} &  b_{22}
\end{pmatrix} .
$$
Matrix forms have the following transformations:

 (1)  Change  {$ {x \to c x}$}: $A_0+A_1x \to A_0+cA_1x$ 

 (2)  Change {$ {Y \to x^{-\gamma} Y}$}\, ($ c=  q^\gamma$): $A_0+A_1x \to c A_0+cA_1x$ 
 
 (3) Change $x \to 1/x$ and $Y \to f(x)Y$, where $f(x)$ is a scalar function
 
 (4) Change {$ {Y \to P Y}$}: $A_0+A_1x \to P^{-1}(A_0+cA_1)P$ 

\noindent 
The equation \eqref{matr} has 
 {eight} parameters $A_0$ and $A_1$, but we may generically deduce to {three}  by the transformations 
(1-4). 

\begin{lemma} 
We assume that $b_{12}=0$ in \eqref{matr}.  
 Then $y_1(x )$ satisfies a {single equation 
of the hypergeometric type}.
\end{lemma} 

\begin{proof}
If we eliminate $y_2$, we obtain
\begin{eqnarray*}
 y_1(xq^2)= [a_{11} + a_{22} + (b_{11}q+b_{22}) x ]  y_1(xq) \hskip 4cm \\ 
 +  [ (a_{12}a_{21} -a_{11}  a_{22} q) 
+ (a_{12} b_{21} -a_{11} b_{22}   -a_{22} b_{11})x - b_{11} b_{22} x^2]y_1(x) =0. 
\end{eqnarray*}
By Lemma \ref{guage}, this equation reduces to the hypergeometric type.
\end{proof}
Since  we may assume $b_{12}=0$ by a transformation (4), this assumption is 
not essential.  
We give a classification theorem for $q$-difference equation of $2\times 2$ 
matrix type:
\begin{theorem} A $q$-difference equation \eqref{matr} 
reduces to one of the following equation by transforms (1-4) except for the 
cases (i) $A_0$  or $A_1$ is a zero matrix, (ii) $\det A(x)\equiv 0$.

\noindent
{1)} $\det A_0 \not= 0, \det A_1 \not= 0$: {$ _2\varphi_1(a,b;c;q,x)$}
$$ A(x)=  \begin{pmatrix}
1 &  (1-a)/c \\
0 &    q/c
\end{pmatrix} +
 \begin{pmatrix}
-b/c &  0 \\
(c-b)q/c &  -aq/c
\end{pmatrix} x.$$ 
{2)} $\det A_0 \not= 0, \det A_1 \sim \mathrm{diag} (0, \mu)$, $\det A(x)\not=\textrm{const.}$: {$ _1\varphi_1(a  ;c;q,x)$}
$$ A(x)=  \begin{pmatrix}
1 &  (1-a)/c \\
0 &    q/c
\end{pmatrix} +
 \begin{pmatrix}
-1/c &  0 \\
-q/c &  0
\end{pmatrix} x.$$ 
3-1) $\det A_0 \not= 0,  A_1^2=1$: {$ E_\nu^{(1)}(x;q)$}
 $$ A(x)=  \begin{pmatrix}
q^\nu &  1 \\
0 &    q^{-\nu}
\end{pmatrix} +
 \begin{pmatrix}
0 &  0 \\
-1/4 &  0
\end{pmatrix} x.$$ 
3-2) $\det A_0 \not= 0,  A_1\sim \mathrm{diag} (0, \mu), \det A(x)=\textrm{const.}$: {$E_\nu^{(3)}(x;q)$}
$$ A(x)=  \begin{pmatrix}
0 &  1 \\
-1 &  q^\nu+ q^{-\nu}
\end{pmatrix} +
 \begin{pmatrix}
0 &  0 \\
0 &  -q^{2-\nu}
\end{pmatrix} x.$$
{3-3)} $ A_0\sim \mathrm{diag} (0, \lambda),  A_1\sim \mathrm{diag} (0, \mu)$: 
{ ${}_1\varphi_1(a ;0; q, x)$}
$$ A(x)=  \begin{pmatrix}
a &  a \\
q - a &  q - a
\end{pmatrix} +
 \begin{pmatrix}
0 &   0 \\
0 &  -q
\end{pmatrix} x.$$
4-1) $\det A_0 \not= 0,  A_1\sim \mathrm{diag} (0, \mu), \det A(x)=\textrm{const.}$: the {$q$-Airy} function
  $$ A(x)=  \begin{pmatrix}
0 &  1 \\
1 &  0
\end{pmatrix} +
 \begin{pmatrix}
0 &  0 \\
0 &  -1
\end{pmatrix} x.$$ 
4-2)  $\det A_0 \sim \mathrm{diag} (0, \lambda),  A_1^2=1$: {the Ramanujan} function
  $$ A(x)=  \begin{pmatrix}
1 &  -q^{-1} \\
0 &  0
\end{pmatrix} +
 \begin{pmatrix}
0 &  0 \\
1 & 0
\end{pmatrix} x.$$ 
5) $A_0^2=1,  A_1^2=1$:
  $$ A(x)=  \begin{pmatrix}
0 & 1 \\
0 &  0
\end{pmatrix} +
 \begin{pmatrix}
0 &  0 \\
1 & 0
\end{pmatrix} x.$$ 
\end{theorem} 

\begin{proof}
It is easily checked by direct calculation. We show a single equation satisfied by $y_1$.

\medskip \noindent
{1)} $y_1$ satisfies
$$y_1(xq^2)-\frac{ c+q-(a+b)qx}{c}y_1(qx)+\frac{q(1-x)(c-ab  x)}{c^2} y_1(x)=0. $$
And $y_1(x)={}_2\varphi_1(a , b;c; q, x)/(abx/c;q)_\infty$ is a special solution. 
A fundamental solution is
\[ Y=
\frac{1}{(abx/c;q)_\infty} \begin{pmatrix}
{}_2\varphi_1(a , b;c; x) &  \frac{1-a}{q-c}x^{1-\gamma} {}_2\varphi_1(aq/c , bq/c; q^2/c; x) \\
\frac{b-c}{1-c} {}_2\varphi_1(aq , b;cq; x)x &  x^{1-\gamma}{}_2\varphi_1(aq/c , b/c; q/c; x) 
\end{pmatrix}.
\]

\medskip \noindent
{2)} $y_1$ satisfies
$$ y_1(xq^2)-\frac{(c+q-qx)}{c}y_1(qx)+\frac{q (c-a  x)}{c^2}y_1(x)=0,$$
And $y_1(x)={}_1\varphi_1(a ;c; q,x)/(a x/c;q)_\infty$ is a special solution. 
A fundamental solution is
\[ Y=
\frac{1}{(a x/c;q)_\infty} \begin{pmatrix}
{}_1\varphi_1(a  ;c; x) &  \frac{1-a}{q-c}x^{1-\gamma} {}_2\varphi_1(aq/c  ; q^2/c; x) \\
\frac{1}{1-c} {}_1\varphi_1(aq ;cq; x)x &  x^{1-\gamma}{}_2\varphi_1(aq/c  ; q/c; x) 
\end{pmatrix}.
\]

\medskip \noindent
{3-1)} $y_1$ satisfies
$$y_1(xq^2)-(q^\nu+ q^{-\nu})y_1(xq)+(1+x/4)y_1(x)=0,$$
which is the equation of $E_\nu^{(1)}(x;q)$.
A fundamental solution is
\[ Y=
\begin{pmatrix}
x^\nu {}_2\varphi_1(0,0;q^{2\nu+1};-x/4) &  
\frac{1}{q^{-\nu }-q^{ \nu }} x^{-\nu} {}_2\varphi_1(0,0;q^{-2\nu+1};-x/4)   \\
\frac{q^{ \nu }}{4(1-q^{2\nu+1})} x^{\nu+1} {}_2\varphi_1(0,0;q^{2\nu+2}; -x/4)  &  
x^{-\nu} {}_2\varphi_1(0,0;q^{-2\nu }; -x/4) 
\end{pmatrix}.
\]

\medskip \noindent
{3-2)} $y_1$ satisfies
$$y_1(xq^2)+ [-(q^\nu+ q^{-\nu}) +q^{2-\nu}x] y_1(xq)+  y_1(x)=0,$$
which is the equation of $E_\nu^{(3)}(x;q)$.
A fundamental solution is
\[ Y=
\begin{pmatrix}
x^\nu \, {}_1\varphi_1\left( 0; q^{1+2\nu };q;  {q^2}x   \right)  &   
 x^{-\nu} \, {}_1\varphi_1\left( 0; q^{1-2\nu};q;  {q^{2-2\nu}}x   \right)   \\
q^{ \nu } x^\nu \, {}_1\varphi_1\left( 0; q^{1+2\nu };q;  {q^3}x   \right)  &  
q^{-\nu } x^{-\nu} \, {}_1\varphi_1\left( 0; q^{1-2\nu};q;  {q^{3-2\nu}}x   \right) 
\end{pmatrix}.
\]

\medskip \noindent
{3-3)} $y_1$ satisfies
$$y_1(xq^2) +q(x-1) y_1(xq )-aq y_1(x)=0,$$
And $y_1(x)={}_1\varphi_1(a ;0; q, x)/\theta_q(-a x)$ is a special solution. 
A fundamental solution is
\[ Y=\frac{1}{\theta_q(-a x)}
\begin{pmatrix}
  {}_1\varphi_1\left( a; 0;q;   x   \right)  &   
\frac{a}{q-a}\theta_q(-a x/q){}_2\varphi_0\left(q/a, 0; -;q;  a x/q^2   \right)   \\
 -{}_1\varphi_1\left( a/q; 0;q;   x q   \right)   &  
 \theta_q(-a x/q){}_2\varphi_0\left(q^2/a, 0; -;q;  a x/q^2   \right) 
\end{pmatrix}.
\]

\medskip \noindent
{4-1)} $y_1$ satisfies
$$y_1(xq^2) +x y_1(xq )-y_1(x)=0,$$
which is the $q$-Airy equation.
A fundamental solution is
\[ Y= 
\begin{pmatrix}
 {}_1\varphi_1(0;-q;q, -x)  &   
e^{\pi i \textrm{lq}\,x} {}{}_1\varphi_1(0;-q;q, x)  \\
 {}_1\varphi_1(0;-q;q, -xq)  &  
- e^{\pi i \textrm{lq}\,  x} {}{}_1\varphi_1(0;-q;q, qx)
\end{pmatrix}.
\]

\medskip \noindent
{4-2)} $y_1$ satisfies
$$y_1(xq^2) -y_1(xq )+ x/q y_1(x)=0.$$
And $y_1(x)={\textrm{A}}_q(x)/\theta_q(x/q)$ is a special solution. 
A fundamental solution is
\[ Y=
\begin{pmatrix}
 \frac{1}{\theta_q(x/q)} {}_0\varphi_1\left( -;0;q; -q x \right)  &   
\frac{1}{x} {}_2\varphi_0\left(0, 0; -;q; - x/q  \right)   \\
 \frac{q}{\theta_q(x/q)} {}_0\varphi_1\left( -;0;q; -  x \right)    &  
{}_2\varphi_0\left(0, 0; -;q; - x/q^2   \right) 
\end{pmatrix}.
\]

\medskip \noindent
5) $y_1$ satisfies
$$y_1(xq^2)  -x y_1(x)=0.$$
And $y_1(x)=1/\theta_{q^2}(x)$ is a special solution. Since $A(xq)A(x)=\textrm{diag}(x, xq)$, 
this case reduces to a $q$-difference equation of the first order. 
A fundamental solution is
\[ Y=
\begin{pmatrix}
1/\theta_{q^2}(x) &   \theta_{q^2}(x/q) /\theta_{q }(x/q) \\
1/\theta_{q^2}(q x) &  \theta_{q^2}(x) /\theta_{q }(x)
\end{pmatrix}.
\] 
\end{proof}

\section{Summary} \label{s:s}

There exist  seven types of  $q$-hypergeometric equations:
$$ {{}_2\varphi_1,\,  _1\varphi_1},\, _2\varphi_1(0,0;c),\,
{ _1\varphi_1(0;c),\,
_1\varphi_1(a;0),\, _1\varphi_1(0;-q)},\, _0\varphi_1(-;0).$$
 {Five} of seven $q$-hypergeometric functions correspond to  
  {particular solutions of $q$-Painlev\'e equations}. 
 { Jackson's first $q$-Bessel} $_2\varphi_1(0,0;c)$ corresponds 
to particular solutions of {$q$-$P_{\textrm {III}}$}, 
which is a symmetric specialization of $q$-$P_{\textrm VI}$.

 In three types of   $q$-Bessel functions,  Hahn-Exton's $q$-Bessel 
$J^{(3)}_\nu (x;q)$  might be a  right  $q$-Bessel function. 
In two types of  $q$-Airy functions,  Hamamoto-Kajiwara-Witte's 
$q$-Airy function  $\textrm{Ai}_q(x)$  might be a right  $q$-Airy function. 
 The Ramanujan function  $_0\varphi_1(-;0)$ is connected to 
the $q$-Airy function by a {shearing transformation} and Morita's connection formula.

\end{document}